\documentclass[a4paper, 12pt]{article}
\pdfoutput=1

\usepackage[T1]{fontenc}
\usepackage[utf8]{inputenc}

\usepackage{amssymb,mathtools,dsfont,microtype}
\usepackage[dvipsnames]{xcolor}
\usepackage[noBBpl]{mathpazo}

\DeclareFontFamily{U} {cmr}{}

\DeclareFontShape{U}{cmr}{m}{n}{
    <-6> cmr5
    <6-7> cmr6
    <7-8> cmr7
    <8-9> cmr8
    <9-10> cmr9
    <10-12> cmr10
    <12-> cmr12}{}

\DeclareSymbolFont{Xcmr} {U} {cmr}{m}{n}
\DeclareMathSymbol{\Omega}{\mathord}{Xcmr}{'012}

\usepackage{enumitem}
\setlist[enumerate]{topsep=0ex,itemsep=0ex,parsep=0.4ex}

\usepackage[left = 2.5cm, right = 2.5cm, top = 2.5cm, bottom = 2.5cm, headsep = 12pt,
    headheight = 15pt]{geometry}
\usepackage{parskip,needspace}

\usepackage[hyphens]{url}
\usepackage[linktoc = all, hidelinks, colorlinks, unicode=true, pagebackref=true]{hyperref}
\hypersetup{linkcolor={blue!70!black}, citecolor={green!70!black}, urlcolor={blue!70!black}}

\usepackage{amsthm}
\theoremstyle{plain}
\newtheorem{thm}{Theorem}[section]
\newtheorem{cor}[thm]{Corollary}
\newtheorem{lem}[thm]{Lemma}
\newtheorem{fact}[thm]{Fact}

\theoremstyle{definition}
\newtheorem{defi}[thm]{Definition}

\DeclareFontFamily{U}{matha}{\hyphenchar\font45}
\DeclareFontShape{U}{matha}{m}{n}{
    <5> <6> <7> <8> <9> <10> gen * matha
    <10.95> matha10 <12> <14.4> <17.28> <20.74> <24.88> matha12
}{}
\DeclareSymbolFont{matha}{U}{matha}{m}{n}

\DeclareMathSymbol{\specialuparrow}{\mathrel}{matha}{"D2}

\renewcommand*{\backref}[1]{}
\renewcommand*{\backrefalt}[4]{
    \ifcase #1 Not cited.\or \(\specialuparrow\)#2\else \(\specialuparrow\)#2\fi}

\renewcommand{\epsilon}{\varepsilon}
\renewcommand{\geq}{\geqslant}
\renewcommand{\leq}{\leqslant}
\renewcommand{\emptyset}{\varnothing}

\DeclarePairedDelimiter{\abs}{\lvert}{\rvert}
\DeclarePairedDelimiter{\ceil}{\lceil}{\rceil}
\DeclarePairedDelimiter{\set}{\lbrace}{\rbrace}

\newcommand*{\bP}{\mathbb{P}}

\newcommand{\defn}[1]{\textcolor{Maroon}{\emph{#1}}}
\newcommand*{\eps}{\varepsilon}

\newcommand*{\Ex}{\text{Ex}}
\newcommand*{\ii}{\mathbf{i}}
\newcommand*{\jj}{\mathbf{j}}
\newcommand*{\kk}{\mathbf{k}}
\newcommand*{\hh}{\mathbf{h}}
\newcommand*{\xx}{\mathbf{x}}
\newcommand*{\yy}{\mathbf{y}}
\newcommand*{\ww}{\mathbf{w}}

\title{Extremal numbers of acyclic matrices}
\author{Lior Gishboliner\thanks{Department of Mathematics, University of Toronto, Canada.
        \emph{Email}:
        \href{mailto:lior.gishboliner@utoronto.ca}{\tt lior.gishboliner@utoronto.ca}.
        Research supported by an NSERC Discovery Grant.
    } \and
    Xiangyu Li\thanks{Department of Mathematics, University of Toronto, Canada.
        \emph{Email}: \href{mailto:xiangyuu.li@mail.utoronto.ca}{\tt
            xiangyuu.li@mail.utoronto.ca}.
        Research supported by a University of Toronto Excellence Award (UTEA).}
}
\date{}

\begin{document}

\maketitle

\begin{abstract}
    For a zero--one matrix \(P\), let \(\Ex(n,P)\) denote the maximum
    number of ones in an \(n\times n\) zero--one matrix avoiding \(P\).
    We prove that
    \[
        \Ex(n,P) \leq n^{1+O_P(1/\log\log n)}
    \]
    for every acyclic \(P\), settling the weak form of the
    Pach--Tardos conjecture.
\end{abstract}

\section{Introduction}\label{sec:introduction}
Let \(P\) be a fixed \(s \times t\) matrix with 0/1 entries. A copy of \(P\) in a 0/1-matrix
\(M\) is an \(s \times t\) submatrix \(M'\) of \(M\) such that \(P\) can be obtained from \(M'\)
by turning some 1's into 0's. For an integer \(n \geq 1\), let \(\Ex(n,P)\) denote the maximum
number of 1-entries in an \(n \times n\) matrix \(M\) with no copy of \(P\). 

Binary matrices are in one-to-one correspondence with ordered bipartite graphs, which are
bipartite graphs with a linear order on each of the two parts. Indeed, the two parts
correspond to the rows and columns of the matrix, with an edge for each 1-entry. Thus,
each \(s \times t\) matrix \(P\) corresponds to an \(s \times t\) ordered bipartite graph.
The matrix \(P\) is called \defn{acyclic} if this graph is a forest.

A central problem in extremal matrix theory is to understand the behavior of
\(\Ex(n,P)\) for acyclic matrices \(P\). The study of these extremal functions
originated in the work of F\"uredi and Hajnal~\cite{FurediHajnal1992}, following
an earlier result of F\"uredi~\cite{Furedi1990}, who used the extremal number of
a particular \(2 \times 3\) matrix to bound the number of unit distances among
points in convex position. They asked whether
\[
    \Ex(n,P)=O(\operatorname{ex}(G_P,n)\log n)
\]
holds for every matrix \(P\), where \(G_P\) is the corresponding
bipartite graph and \(\operatorname{ex}(G_P,n)\) is its unordered extremal number.
For acyclic \(P\), \(G_P\) is a forest, so \(\operatorname{ex}(G_P,n)=O(n)\),
giving the conjectured bound \(\Ex(n,P)=O(n\log n)\).
Pach and Tardos~\cite{PachTardos2006} later disproved the general conjecture using
patterns whose underlying graphs are even cycles, but the \(O(n\log n)\) bound
for all acyclic \(P\) remained open. 
In the same paper, Pach and Tardos suggested the weaker conjecture that 
\(\Ex(n,P)=O(n\log^{O(1)}n)\) for all acyclic $P$ (see \cite[Conjecture~1]{PachTardos2006}) . Furthermore, they developed a series of reductions
establishing the conjecture for many acyclic matrices. They also proposed the still weaker
bound \(\Ex(n,P)\leq n^{1+o(1)}\) (for all acyclic $P$).

This problem has since attracted considerable interest. Pettie~\cite{Pettie2011a}
disproved the F\"uredi--Hajnal conjecture by constructing an
acyclic matrix \(P\) with 
\(\Ex(n,P)=\Omega(n\log n\log\log n)\).
Pettie and Tardos~\cite{PettieTardos2024} later strengthened this by finding,
for every positive integer \(k\), an acyclic matrix \(P\) with
\(\Ex(n,P)=\Omega(n(\log n)^k)\). Their subsequent work~\cite{PettieTardos2025}
established the lower bound \(\Ex(n,P)\geq n2^{\Omega(\sqrt{\log n})}\) for
some acyclic patterns \(P\), thus disproving the original Pach--Tardos conjecture.
They then reformulated the conjecture, proposing the bounds
\(\Ex(n,P)\leq n^{1+o(1)}\) (weak form) and
\(\Ex(n,P)\leq n2^{O_P(\sqrt{\log n})}\) (strong form)
(see \cite[Conjecture~4.1]{PettieTardos2025}). The strongest prior progress
towards the weak form was due to Kor\'andi, Tardos, Tomon, and
Weidert~\cite{KorandiTardosTomonWeidert2019}, who proved that
\(\Ex(n,P)\leq n2^{(\log n)^{1-C_P}}\), for a large
family of acyclic patterns, including all those with at most three rows. Their work also implies \(\Ex(n,P)\leq n2^{O(\sqrt{\log n})}\) for one of the patterns
in \cite{PettieTardos2025}, matching the lower bound up to constants in the exponent.

Beyond this conjecture, the study of \(\Ex(n,P)\) has revealed rich connections
to other areas of combinatorics. F\"uredi and Hajnal~\cite{FurediHajnal1992}
used the connection with Davenport--Schinzel sequences to exhibit a matrix \(P\)
with \(\Ex(n,P)=\Theta(n\alpha(n))\), where \(\alpha(n)\) is the inverse Ackermann
function, which tends to infinity extremely slowly. More generally, a result of Klazar (see \cite[Theorem~1.1]{Keszegh2009}) gives
\(\Ex(n,P)\leq n2^{\alpha(n)^{O_P(1)}}\) for every \defn{light} pattern \(P\), i.e.,
a pattern with exactly one \(1\) in each column. This connection was further explored
by Pettie~\cite{Pettie2011b}. F\"uredi and Hajnal also conjectured
that \(\Ex(n,P)=O(n)\) whenever \(P\) is a permutation matrix. Marcus and
Tardos~\cite{MarcusTardos2004} proved this conjecture, and their result, via
an earlier reduction by Klazar~\cite{Klazar2000}, also settled the celebrated
Stanley--Wilf conjecture on pattern-avoiding permutations. For further background
on extremal ordered problems, we refer the reader to the surveys of
Tardos~\cite{Tardos2019a,Tardos2019b} and the recent
works~\cite{JanzerJanzerMagnanMethuku2024,JJMT,KucheriyaTardos2023}.

In this paper, we prove the weak Pach--Tardos conjecture. We state and prove
this result in the language of ordered bipartite graphs (instead of matrices).
Thus, the forbidden structure is an ordered bipartite forest, and we may
assume, by adding edges, that this forest is a tree.

\begin{thm}\label{thm:main}
    For all integers \(s,t \geq 1\) and \(\eps > 0\), there exists
    \[
        n_0 \leq 2^{2^{O_{s,t}(1/\eps)}}
    \]
    such that the following holds. Every \(n \times n\) ordered bipartite graph \(G\) with \(n
    \geq n_0\) and \(e(G) \geq n^{1+\eps}\) contains a copy of every \(s \times t\) ordered
    bipartite tree.
\end{thm}

This theorem can be restated as follows:
\begin{cor}
    For every acyclic pattern \(P\), it holds that
    \[
        \Ex(n,P) \leq n^{1 + O_P(\frac{1}{\log\log n})}.
    \]
\end{cor}

\section{Preliminaries}\label{sec:entropy-preliminaries}
We recall some basic facts from information theory. Throughout the paper, all logarithms are
assumed to be base 2, unless stated otherwise. Let \(Z\) be a random variable taking values
\(z_1,\dots,z_n\) with probabilities \(p_1,\dots,p_n\), respectively. The \defn{entropy} of
\(Z\) is defined as
\[
    H[Z] := \sum_{i=1}^n p_i \log \left( \frac{1}{p_i} \right),
\]
with the convention that \(0 \cdot \log(\frac{1}{0}) = 0\).

\begin{fact}\label{fact:entropy basic}
    \hfill
    \begin{enumerate}
        \item If \(Z\) takes \(n\) different values then \(H[Z] \leq \log n\).
        \item If \(Z\) takes each value with probability at most \(p\), then \(H[Z] \geq
            \log(1/p)\).
    \end{enumerate}
\end{fact}

For a random variable \(Z\) and an event \(E\), we write \(H[Z \; | \; E]\) to mean the entropy
of the random variable \(Z\) conditioned on \(E\) (i.e., we condition the probability measure on
\(E\)). In particular, for two random variables \(Z,W\) on the same probability space, we use
\(H[Z \; | \; W = w]\) to denote the entropy of the random variable \(Z\) conditioned on the
event \(W = w\). The \defn{conditional entropy} \(H[Z \; | \; W]\) is defined as
\[
    H[Z \; | \; W] := \sum_{w} \bP[W = w] \cdot H[Z \; | \; W = w],
\]
where the sum is over all values \(w\) taken by \(W\).

We note that \(H[Z \; | \; E]\) is not the same as \(H[Z \; | \; \mathds{1}_E] = \bP[E] \cdot
H[Z \; | \; E] + \bP[\overline{E}] \cdot H[Z \; | \; \overline{E}]\). More generally, for random
variables \(Z,W\) and an event \(E\), we may write \(H[Z \; | \; E, W]\) to mean the conditional
entropy \(H[Z \; | \; W]\) calculated with respect to the probability measure conditioned on
\(E\). Again, this is not the same as \(H[Z \; | \; \mathds{1}_E, W]\).

\begin{fact}\label{fact:H[Z | f(W)]}
    For random variables \(Z,W\) and a function \(f\), it holds that \(H[Z \; | \; f(W)] \geq
    H[Z \; | \; W]\). In particular, \(H[Z] \geq H[Z \; | \; W]\).
\end{fact}

\begin{fact}[chain rule]\label{fact:chain rule}
    Let \(Z_1,\dots,Z_m,W\) be random variables on the same probability space. Then
    \[
        H[(Z_1,\dots,Z_m) \; | \; W] = \sum_{i=1}^m H[Z_i \; | \; Z_1,\dots,Z_{i-1},W].
    \]
\end{fact}

\section{The main lemmas}\label{sec:main-lemmas}

We will assume throughout that \(s,t \geq 2\) (it clearly suffices to prove Theorem
\ref{thm:main} under this assumption). A key ingredient in the proof of Theorem \ref{thm:main}
is the so-called \defn{\(T\)-branching random walk}, which can be used to show that all trees
satisfy the assertion of Sidorenko's conjecture (see \cite{ConlonKimLeeLee2018}). For a tree \(T\) and a graph
\(G\), the \(T\)-branching random walk generates a random homomorphic copy of \(T\) in \(G\) by
embedding the vertices randomly following an order in which every parent precedes its children.
Since we will consider (ordered) bipartite \(T,G\), we adapt the definition to this setting.

\begin{defi}[\(T\)-branching random walk]\label{def:T-branching random walk}
    Let \(T\) be a tree with bipartition \(A,B\), and let \(G\) be a bipartite graph with
    bipartition \(X,Y\) and $e(G) > 0$. The \(T\)-branching random walk on \(G\) is a random homomorphism
    \(\varphi : T \rightarrow G\) generated as follows: root \(T\) at an arbitrary vertex \(v_1
    \in B\),\footnote{We could equally well root \(T\) at a vertex of \(A\).} and fix an
    ordering \(v_1,\dots,v_{\abs{T}}\) of the vertices of \(T\) where for each \(2 \leq h \leq
    \abs{T}\), the parent of \(v_h\) is in \(\set{v_1,\dots,v_{h-1}}\).
    \begin{enumerate}
        \item Choose \(\varphi(v_1) \in Y\) randomly, proportionally to the vertex degrees.
            Namely, for each \(y \in Y\), set \(\varphi(v_1) = y\) with probability
            \(\frac{d(y)}{e(G)}\).
        \item For each \(h = 2,\dots,\abs{T}\), let \(\ell \leq h-1\) be such that \(v_{\ell}\)
            is the parent of \(v_h\), and choose \(\varphi(v_h)\) uniformly at random among the
            neighbors of \(\varphi(v_{\ell})\).
    \end{enumerate}
\end{defi}
Note that requiring \(\varphi(v_1) \in Y\) guarantees that \(\varphi(A) \subseteq X\) and
\(\varphi(B) \subseteq Y\). The following are basic properties of the \(T\)-branching random
walk (see \cite[Section~5.5]{Zhao2023}). For completeness, we include a proof.
\begin{lem}\label{lem:T-branching random walk}
    Let \(\varphi : T \rightarrow G\) be the \(T\)-branching random walk and let
    \(v_1,\dots,v_{\abs{T}}\) be the vertex ordering used to define \(\varphi\).
    \begin{enumerate}
        \item Let \(2 \leq h \leq \abs{T}\), and let \(1 \leq \ell \leq h-1\) be such that
        \(v_{\ell}\) is the parent of \(v_h\) in \(T\). Then, for all
        \(z_1,\dots,z_{h-1},z \in V(G)\) such that the event
        \(\varphi(v_1)=z_1,\dots,\varphi(v_{h-1})=z_{h-1}\)
        has positive probability, we have
        \[
            \bP[\varphi(v_h) = z
            \; | \; \varphi(v_1) = z_1,\dots,\varphi(v_{h-1}) = z_{h-1}] =
            \bP[\varphi(v_h) = z \; | \; \varphi(v_{\ell}) = z_{\ell}].
        \]
        \item For every \(1 \leq h \leq \abs{T}\), if the \(v_1\)--\(v_h\) path has odd (resp.
            even) length, then for every \(u \in X\) (resp. \(u \in Y\)),
        \[
            \bP[\varphi(v_h) = u] = \frac{d(u)}{e(G)}.
        \]
        \item For every edge \(e \in E(T)\), the image \(\varphi(e)\) is distributed uniformly
            on \(E(G)\).
    \end{enumerate}
\end{lem}
\begin{proof}
    Item 1 is immediate from the definition. We now prove Item 2 by induction on \(h\). The case
    \(h=1\) follows from the definition. For \(h \geq 2\), let \(v_{\ell}\) be the parent of
    \(v_h\). By induction and the definition,
    \begin{align*}
        \bP[\varphi(v_h) = u]
        &= \sum_{w \in N(u)}
        \bP[\varphi(v_{\ell}) = w] \cdot
        \bP[\varphi(v_h) = u \; | \; \varphi(v_{\ell}) = w]\\
        &= \sum_{w \in N(u)} \frac{d(w)}{e(G)} \cdot \frac{1}{d(w)}
        = \frac{d(u)}{e(G)}.
    \end{align*}
    Finally we prove Item 3. Let \(e = v_{\ell}v_h\), where \(v_{\ell}\) is the parent of
    \(v_h\), and let \(wz \in E(G)\), with \(w \in X, z \in Y\) (resp. \(w \in Y, z \in X\)) if
    the \(v_1\)--\(v_{\ell}\) path has odd (resp. even) length. By Item 2,
    \begin{align*}
        \bP[\varphi(v_h) = z, \varphi(v_{\ell}) = w]
        &= \bP[\varphi(v_h) = z \; | \; \varphi(v_{\ell}) = w] \cdot
        \bP[\varphi(v_{\ell}) = w]\\
        &= \frac{1}{d(w)} \cdot \frac{d(w)}{e(G)}
        = \frac{1}{e(G)}. \qedhere
    \end{align*}
\end{proof}

Now, suppose that \(T\) is an ordered bipartite tree with parts \(A = \set{a_1<\dots<a_s}\) and
\(B = \set{b_1 < \dots < b_t}\), and that \(G\) is an ordered bipartite graph with parts
\(X,Y\). Suppose further that \(X,Y\) are partitioned into sets \(X_1 < \dots < X_s\) and \(Y_1
< \dots < Y_t\), respectively. If \(\varphi(a_i) \in X_i\) for each \(i \in [s]\) and
\(\varphi(b_j) \in Y_j\) for each \(j \in [t]\) then \(\varphi\) gives an (ordered) copy of
\(T\) in \(G\). Conversely, the (ordered) \(T\)-freeness of \(G\) imposes a restriction on the
distribution of the target sets (among \(X_1,\dots,X_s,Y_1,\dots,Y_t\)) for the images of
\(a_1,\dots,a_s,b_1,\dots,b_t\). Intuitively, this means that, for typical vertices \(x \in X\) and \(y \in Y\),
the distributions of the neighbors of \(x\) among \(Y_1,\dots,Y_t\) and of
\(y\) among \(X_1,\dots,X_s\) cannot simultaneously be close to uniform. This is quantified in the
following lemma using entropy.

\begin{lem}\label{lem:embedding}
    Let \(T\) be an ordered bipartite tree with parts \(A = \set{a_1<\dots<a_s}\) and \(B =
    \set{b_1 < \nolinebreak \dots < \nolinebreak b_t}\). Let \(G\) be a nonempty \(T\)-free ordered
    bipartite graph with parts \(X,Y\) and with a partition \(X_1 < \dots < X_s\) of \(X\) and a
    partition \(Y_1 < \dots < Y_t\) of \(Y\). Choose a uniformly random edge \(e = (\xx,\yy) \in
    E(X,Y)\), and let \(\ii \in [s], \jj \in [t]\) be such that \(\xx \in X_{\ii}\), \(\yy \in
    Y_{\jj}\). Then
    \begin{equation}\label{eq:embedding entropy bound}
        s \cdot H[\ii \; | \; \yy] + t \cdot H[\jj \; | \; \xx] \leq \log(s^st^t-1).
    \end{equation}
\end{lem}
\begin{proof}
    Let \(\varphi : T \rightarrow G\) be the \(T\)-branching random walk. For each \(k \in
    [s]\), let \(\ii_k \in [s]\) be the random index satisfying \(\varphi(a_k) \in X_{\ii_k}\),
    and similarly, for each \(\ell \in [t]\), let \(\jj_{\ell} \in[t]\) be the random index
    satisfying \(\varphi(b_{\ell}) \in Y_{\jj_{\ell}}\). First, observe that
    \begin{equation}\label{eq:embedding entropy upper bound}
        H[(\ii_1,\dots,\ii_s;\jj_1,\dots,\jj_t)] \leq \log(s^st^t - 1).
    \end{equation}
    Indeed, as \(G\) is \(T\)-free, \((\ii_1,\dots,\ii_s;\jj_1,\dots,\jj_t) \neq
    (1,\dots,s;1,\dots,t)\) (since this would give an (ordered) copy of \(T\) in \(G\)), and so
    this vector takes at most \(s^st^t-1\) different values.

    To complete the proof, we lower-bound the LHS in \eqref{eq:embedding entropy upper bound} by
    using the chain rule. To that end, for a vertex \(v \in V(T) = A \cup B\), it will be
    convenient to write
    \[
        \hh(v) :=
        \begin{cases}
            \ii_k & v = a_k, \\
            \jj_{\ell} & v = b_{\ell}.
        \end{cases}
    \]
    Let \(v_1,\dots,v_{s+t}\) be the ordering in which the vertices of \(T\) are processed in
    the definition of the \(T\)-branching random walk (Definition
    \ref{def:T-branching random walk}); recall that \(v_1 \in B\). Since
    \((\hh(v_1),\dots,\hh(v_{s+t}))\) is a reordering of
    \((\ii_1,\dots,\ii_s;\jj_1,\dots,\jj_t)\), we have
    \[
        H[(\ii_1,\dots,\ii_s;\jj_1,\dots,\jj_t)]
        =
        H[(\hh(v_1),\dots,\hh(v_{s+t}))].
    \]
    Hence, by the chain rule (Fact \ref{fact:chain rule}), the LHS of
    \eqref{eq:embedding entropy upper bound} is equal to
    \begin{equation}\label{eq:embedding chain rule}
        H\big[ \big( \hh(v_1),\dots,\hh(v_{s+t}) \big) \big] =
        \sum_{h=1}^{s+t} H \big[ \hh(v_h) \; \big| \; \hh(v_1),\dots,\hh(v_{h-1}) \big].
    \end{equation}
    We now consider each term on the RHS above. First, consider \(h=1\). Observe that for a
    uniformly random edge \(\xx\yy \in E(G)\) and for each \(y \in Y\), we have \(\bP[\yy = y] =
    \frac{d(y)}{e(G)}\), which is exactly the probability that \(\varphi(v_1) = y\) (see
    Definition \ref{def:T-branching random walk}). Hence, the index \(\hh(v_1)\) has the same
    distribution as the index \(\jj\) from the statement of the lemma. Therefore, \(H[\hh(v_1)]
    = H[\jj] \geq H[\jj \; | \; \xx]\).

    Next, let \(2 \leq h \leq s+t\), and let \(v_{\ell}\) be the parent of \(v_h\). Let \(f\)
    map each tuple \((z_1,\dots,z_{h-1})\) of vertices to the tuple of their respective block
    indices. Then Fact \ref{fact:H[Z | f(W)]} applied to \(f\) gives
    \[
        H \big[ \hh(v_h) \; \big| \; \hh(v_1),\dots,\hh(v_{h-1}) \big] \geq
        H \big[ \hh(v_h) \; \big| \; \varphi(v_1),\dots,\varphi(v_{h-1}) \big]
    \]
    and by Item 1 of Lemma \ref{lem:T-branching random walk},
    \[
        H \big[ \hh(v_h) \; \big| \; \varphi(v_1),\dots,\varphi(v_{h-1}) \big] =
        H \big[ \hh(v_h) \; \big| \; \varphi(v_{\ell}) \big].
    \]
    Furthermore, by Item 3 of Lemma \ref{lem:T-branching random walk},
    \((\varphi(v_{\ell}),\varphi(v_h))\) is distributed as a uniformly random edge of \(G\).
    Therefore, if \(v_h \in A\) then \(v_{\ell} \in B\) and \(H[ \hh(v_h) \; | \;
    \varphi(v_{\ell})] = H[\ii \; | \; \yy]\), and if \(v_h \in B\) then \(v_{\ell} \in A\) and
    \(H[ \hh(v_h) \; | \; \varphi(v_{\ell})] = H[\jj \; | \; \xx]\). Summarizing, for each \(1
    \leq h \leq s+t\) we have
    \[
        H \big[ \hh(v_h) \; \big| \; \hh(v_1),\dots,\hh(v_{h-1}) \big] \geq
        \begin{cases}
            H[\ii \; | \; \yy] & v_h \in A, \\
            H[\jj \; | \; \xx] & v_h \in B.
        \end{cases}
    \]
    Summing this over all \(1 \leq h \leq s+t\) gives
    \[
        \sum_{h=1}^{s+t} H \big[ \hh(v_h) \; \big| \; \hh(v_1),\dots,\hh(v_{h-1}) \big] \geq
        s \cdot H[\ii \; | \; \yy] + t \cdot H[\jj \; | \; \xx].
    \]
    Combining the above with \eqref{eq:embedding entropy upper bound} and
    \eqref{eq:embedding chain rule} proves the lemma.
\end{proof}

In light of Lemma \ref{lem:embedding}, in order to find a copy of \(T\) in \(G\) it suffices to
find sets \(X,Y\) (in the respective parts of \(G\)) with partitions \(X = X_1 \cup \dots \cup
X_s\) and \(Y = Y_1 \cup \dots \cup Y_t\) such that the distribution of edges between \(X,Y\)
violates \eqref{eq:embedding entropy bound}. This is achieved by our second main lemma:

\begin{lem}\label{lem:main}
    For all integers \(s,t \geq 2\), \(\epsilon > 0\), and \(0 < \delta \leq 1\), there exists
    \[
        n_0 = n_0(s,t,\eps,\delta) \leq 2^{2^{O_{s,t}(\frac{1}{\eps\delta})}}
    \]
    such that the following holds. Let \(G\) be an \(n \times n\) ordered bipartite graph with
    parts \(U,V\), \(n \geq n_0\), and \(e(G) \geq n^{1+\eps}\). Then there exist subsets \(X_1
    < \dots < X_s\) of \(U\), subsets \(Y_1 < \dots < Y_t\) of \(V\), and a subgraph \(G'\) of
    \(G[X,Y]\), where \(X := X_1 \cup \dots \cup X_s\) and \(Y := Y_1 \cup \dots \cup Y_t\),
    such that the following holds: Choose a uniformly random edge \((\xx,\yy) \in E_{G'}(X,Y)\)
    and let \(\ii \in [s], \jj \in [t]\) be such that \(\xx \in X_{\ii}\), \(\yy \in Y_{\jj}\).
    Then
    \begin{equation}\label{eq:main lemma}
        H[ \ii \; | \; \yy ] = \log s,
        \qquad
        H[ \jj \; | \; \xx ] \geq \log t - \delta.
    \end{equation}
\end{lem}

Note that the equality \(H[\ii\mid\yy]=\log s\) means that, for every nonisolated \(y\in Y\),
the conditional distribution of \(\ii\) given \(\yy=y\) is uniform. Equivalently, every \(y\in
Y\) has equally many neighbors in each of the sets \(X_1,\dots,X_s\) in \(G'\).

Theorem \ref{thm:main} easily follows from Lemmas \ref{lem:embedding} and \ref{lem:main}:

\begin{proof}[Proof of Theorem \ref{thm:main}]
    Set
    \[
        \delta := \frac{1}{2t} \cdot \log \left( \frac{s^st^t}{s^s t^t - 1} \right) =
        \Omega_{s,t}(1),
    \]
    and let \(n_0 = n_0(s,t,\eps,\delta)\) be given by Lemma \ref{lem:main}. Let \(n \geq n_0\),
    and let \(G\) be an \(n \times n\) ordered bipartite graph with parts \(U,V\) and \(e(G)
    \geq n^{1+\eps}\). Let \(X_1 < \dots < X_s\), \(Y_1 < \dots < Y_t\), and \(G' \subseteq
    G[X,Y]\) be given by Lemma \ref{lem:main}, where \(X := X_1 \cup \dots \cup X_s\) and \(Y :=
    Y_1 \cup \dots \cup Y_t\). Lemma \ref{lem:main} guarantees that if \((\xx,\yy) \in
    E_{G'}(X,Y)\) is chosen uniformly at random and \(\ii \in [s], \jj \in [t]\) are the unique
    indices with \(\xx \in X_{\ii}\), \(\yy \in Y_{\jj}\), then \(H[ \ii \; | \; \yy ] = \log
    s\) and \(H[ \jj \; | \; \xx ] \geq \log t - \delta\). This means that
    \[
        s \cdot H[\ii \; | \; \yy] + t \cdot H[\jj \; | \; \xx] \geq s \log s + t\log t -
        t\delta =
        \log \left( s^st^t \right) - t\delta > \log \left( s^st^t - 1 \right),
    \]
    using the choice of \(\delta\). The contrapositive of Lemma \ref{lem:embedding} now states
    that \(G'\) (and hence also \(G\)) contains a copy of every \(s \times t\) ordered bipartite
    tree \(T\).
\end{proof}

\section{Proof of Lemma \ref{lem:main}}\label{sec:proof-main-lemma}
First, we prove a lemma which shows that in a sufficiently dense ordered bipartite graph \(G\)
with parts \(U,V\), we can find nonempty subsets \(X_1<\dots<X_s\) of \(U\) and retain a fixed
positive fraction of the edges of \(G[X,V]\), where \(X:=X_1\cup\dots\cup X_s\), so that every
\(v\in V\) has the same number of neighbors in each \(X_i\).

\begin{lem}\label{lem:neighborhood splitting}
    For every integer \(s\geq2\) and every \(\eps > 0\), there exists
    \[
        m=m(s,\eps)\leq s^{O(s/\eps)}
    \]
    such that the following holds. If \(G\) is an ordered bipartite graph with parts \(U,V\)
    satisfying \(p := \abs{U}\geq m\), \(V\neq\emptyset\), and \(e(G)\geq p^\eps\abs{V}\), then
    there are nonempty subsets \(X_1<\dots<X_s\) of \(U\) such that, setting
    \(X:=X_1\cup\dots\cup X_s\), we have
    \begin{equation}\label{eq:neighborhood splitting}
        \sum_{v\in V}\min_{1\leq i\leq s}d_{X_i}(v)
        \geq\frac1m\max\left(
        \frac{e(G)\abs{X}}{p},\;e_G(X,V)
        \right).
    \end{equation}
\end{lem}

Thus, \eqref{eq:neighborhood splitting} guarantees that we can delete edges from \(G[X,V]\) so
that every \(v\in V\) has exactly \(\min_{1\leq i\leq s}d_{X_i}(v)\) neighbors in each \(X_i\),
while retaining at least a \(1/m\)-fraction of the edges of \(G[X,V]\) and at least
\(e(G)|X|/(mp)\) edges.

The proof of Lemma \ref{lem:neighborhood splitting} works by taking a sequence of nested
partitions of \(U\), repeatedly dividing each part into at most \(d\) intervals. This structure
can be represented by a rooted \(d\)-ary tree, whose leaves correspond to the vertices of \(U\)
(strictly speaking, \(U\) will only correspond to a subset of the leaves, since \(p = \abs{U}\)
need not be an integer power of \(d\)). To handle such trees, it is convenient to introduce the
following notation: Let \([d]^{\leq \ell}\) denote the set of strings of length at most \(\ell\)
over the alphabet \([d]\) (this includes the empty string). For strings \(\sigma,\tau \in
[d]^{\leq \ell}\), we write \(\sigma \preceq \tau\) to mean that \(\sigma\) is a prefix of
\(\tau\). In particular, \(\sigma i\) denotes the string obtained from \(\sigma\) by appending
the character \(i\), where \(i \in [d]\). We also use \(\sigma \prec \tau\) to mean that
\(\sigma \preceq \tau\) and \(\sigma \neq \tau\). Note that the nodes of the rooted \(d\)-ary
tree correspond to strings \(\sigma \in [d]^{\leq \ell}\). The children of a non-leaf node
\(\sigma \in [d]^{< \ell}\) are \(\sigma i\), \(i \in [d]\). A node \(\sigma\) being an ancestor
of a node \(\tau\) is the same as \(\sigma \preceq \tau\).

\begin{proof}[Proof of Lemma \ref{lem:neighborhood splitting}]
    Choose an integer \(d \geq s\) large enough so that
    \[
        s d^{4/d} \leq d^{\varepsilon/4}
        \qquad
        \text{and}
        \qquad
        d^{s\varepsilon/2} \geq 4.
    \]
    It is easy to see that one can take \(d \leq s^{O(1/\varepsilon)}\). Next, set
    \[
        m := d^s, \qquad \ell:=\ceil*{\log_d p}.
    \]
    Then \(m\leq s^{O(s/\eps)}\) and \(p\geq m\geq d\) (where \(p \geq m\) holds by assumption
    and \(m \geq d\) by the choice of \(m\)).

    Suppose, for a contradiction, that the conclusion of the lemma fails. Since \(d^\ell\geq
    p\), we may choose an order-preserving injection \(\lambda\colon U\to[d]^\ell\), where
    \([d]^\ell\) is ordered lexicographically. For \(\sigma\in[d]^{\leq\ell}\), let
    \(U_\sigma:=\{u\in U:\sigma\preceq\lambda(u)\}\). At each level, the nonempty sets
    \(U_\sigma\) partition \(U\) into intervals.

    For \(v\in V\) and \(\sigma\in[d]^{<\ell}\), let \(I_{v,\sigma}\subseteq[d]\) consist of the
    \(s-1\) indices \(i \in [d]\) for which \(d_{U_{\sigma i}}(v)\) is largest, breaking ties
    arbitrarily. Define \(g_\sigma\colon E(G)\to\{0,1\}\) by \(g_\sigma(uv)=1\) if \(u\in
    U_{\sigma j}\) for some \(j\in[d]\setminus I_{v,\sigma}\), and \(g_\sigma(uv)=0\) otherwise.
    In other words, \(g_{\sigma}(uv) = 1\) if \(u \in U_{\sigma}\) and \(u\) is not among the
    \(s-1\) heaviest children of \(\sigma\) with respect to \(v\). We regard \(g_\sigma(uv)=1\)
    as a mark on \(uv\) at the node \(\sigma\). Let \(g:=\sum_{\sigma\in[d]^{<\ell}}g_\sigma\),
    so that \(g(uv)\) is the total number of marks received by \(uv\).

    For now, fix any \(\sigma\in[d]^{<\ell}\). First observe that
    \begin{equation}\label{eq:neighborhood splitting aux 1}
        \sum_{uv\in E(G)}g_\sigma(uv)
        =
        \sum_{v\in V}
        \sum_{j \in [d] \setminus I_{v,\sigma}}d_{U_{\sigma j}}(v)
        \leq
        \sum_{v\in V}
        \sum_{\substack{J\subseteq[d]\\|J|=s}}
        \min_{i\in J}d_{U_{\sigma i}}(v).
    \end{equation}
    Here, the inequality holds because for every \(j \in [d] \setminus I_{v,\sigma}\), for the
    set \(J := I_{v,\sigma} \cup \{j\}\) we have \( d_{U_{\sigma j}}(v) = \min_{i \in J}
    d_{U_{\sigma i}}(v) \) (by the definition of \(I_{v,\sigma}\)), and so the term
    \(d_{U_{\sigma j}}(v)\) is accounted for on the RHS of the inequality. Next, fix any
    \(J\subseteq[d], |J|=s\) and put \(X := \bigcup_{i \in J} U_{\sigma i}\). By assumption, the
    conclusion of the lemma fails for the \(s\) sets \((X_1,\dots,X_s) = (U_{\sigma i} : i \in
    J)\), and therefore
    \[
        \sum_{v\in V} \min_{i\in J}d_{U_{\sigma i}}(v) \leq
        \frac1m
        \max\left(
        \frac{e(G)\abs{X}}{p},\;e_G(X,V)
        \right)
        \leq
        \frac1m \sum_{i\in J}
        \left(\frac{e(G)|U_{\sigma i}|}{p}
        +e_G(U_{\sigma i},V)\right).
    \]
    More precisely, the failure of the conclusion of the lemma implies the above if \(U_{\sigma
    i} \neq \varnothing\) for all \(i \in J\) (since the lemma requires nonempty
    \(X_1,\dots,X_s\)), whereas if one of these sets is empty then the above is immediate.
    Combining this with \eqref{eq:neighborhood splitting aux 1}, we get
    \begin{align}\label{eq:neighborhood splitting aux 2}
        \sum_{uv\in E(G)}g_\sigma(uv) &\leq \nonumber
        \frac1m\sum_{\substack{J\subseteq[d]\\|J|=s}}
        \sum_{i\in J}
        \left(\frac{e(G)|U_{\sigma i}|}{p}
        +e_G(U_{\sigma i},V)\right)\\
        & = \frac{\binom{d-1}{s-1}}{m}
        \left(\frac{e(G)|U_\sigma|}{p}+e_G(U_\sigma,V)\right)\nonumber \\
        & \leq
        \frac1d
        \left(\frac{e(G)|U_\sigma|}{p}+e_G(U_\sigma,V)\right),
    \end{align}
    where the last inequality uses the choice of \(m\).

    Observe that for every \(0 \leq k \leq \ell-1\), the sets \(U_\sigma\), \(\sigma\in[d]^k\),
    partition \(U\). We now use this to bound the total number of marks, as follows.
    \begin{align*}
        \sum_{uv\in E(G)}g(uv)
        &= \sum_{\sigma\in[d]^{<\ell}}
        \sum_{uv\in E(G)}g_\sigma(uv)
        \; \overset{\eqref{eq:neighborhood splitting aux 2}}{\leq} \;
        \frac1d\sum_{k=0}^{\ell-1}\sum_{\sigma\in[d]^k}
        \left(\frac{e(G)|U_\sigma|}{p}+e_G(U_\sigma,V)\right)\\
        &= \frac{\ell}{d}
        \left(\frac{e(G)}p\cdot p+e(G)\right)
        = \frac{2\ell}{d}\,e(G).
    \end{align*}
    Thus, by Markov, the number of edges \(uv\in E(G)\) with \(g(uv)>4\ell/d\) is at most
    \[
        \frac{d}{4\ell}\sum_{uv\in E(G)}g(uv)
        \leq \frac12e(G).
    \]
    Consequently, at least \(\frac12e(G)\) edges receive at most \(4\ell/d\) marks.

    We now bound the number of such edges from above. Fix \(v\in V\). For each edge \(uv\),
    define its \defn{address} \(\mu(uv)\in\{1,\dots,s-1,\star\}^{\ell}\) as follows. For each
    \(k\in[\ell]\), let \(\sigma\) be the prefix of \(\lambda(u)\) of length \(k-1\). Set
    \(\mu(uv)_k=i\) if \(\lambda(u)_k\) is the \(i\)-th smallest element of \(I_{v,\sigma}\),
    and \(\mu(uv)_k=\star\) otherwise. This address has exactly \(g(uv)\) stars. There are
    \(s^\ell\) possible addresses, and an address with \(r\) stars represents at most \(d^r\)
    edges incident to \(v\). Indeed, if \(\mu\) is an address and \(K\) is the set of indices
    \(k \in [\ell]\) with \(\mu_k = \star\), then a vertex \(u\) with \(\mu(uv) = \mu\) is
    determined uniquely by \((\lambda(u)_k : k \in K)\), for which there are \(d^{|K|}\)
    options. We conclude that at most \(s^\ell d^{4\ell/d}\) edges incident to each \(v\)
    receive at most \(4\ell/d\) marks.

    We now show that this contradicts the lower bound on the number of such edges. First, our
    parameter choices give \(s d^{4/d} \leq d^{\varepsilon/4}\), \(d^\ell\leq dp\leq p^2\), and
    \(p^{\eps/2}\geq m^{\eps/2}= d^{s\varepsilon/2} \geq 4\). Together with the assumption
    \(e(G)\geq p^\eps|V|\), this gives
    \[
        \frac12e(G)
        \leq |V|s^\ell d^{4\ell/d}
        \leq |V|d^{\eps\ell/4}
        \leq |V|p^{\eps/2}
        \leq \frac14e(G),
    \]
    a contradiction.
\end{proof}

Our proof of Lemma \ref{lem:main} starts by applying Lemma \ref{lem:neighborhood splitting}
twice. First, we apply it to \(V\), extracting nonempty subsets \(W_1<\dots<W_q\) and deleting
edges so that each \(u\in U\) has the same degree into each of \(W_1,\dots,W_q\). The guarantees
of Lemma \ref{lem:neighborhood splitting} (namely, \eqref{eq:neighborhood splitting}) permit a
second application, now to \(U\), giving \(X_1<\dots<X_s\). Once again, we delete edges so that
each \(w \in W\) has the same degree into each \(X_i\). If an edge \((\xx,\ww) \in E(X,W)\) is
chosen uniformly at random and \(\ii \in [s], \kk \in [q]\) are such that \(\xx \in X_{\ii}, \ww
\in W_{\kk}\), then \(\ii\) is distributed uniformly conditioned on \(\ww\). At the same time,
as we shall see, the fact that we only lose a factor of \(1/m\) in an application of Lemma
\ref{lem:neighborhood splitting} (due to \eqref{eq:neighborhood splitting}) means that \(\kk\)
conditioned on \(\xx\) is still close to uniform in entropy: \(H[\kk \mid \xx]\geq\log q-\log
m\). However, this falls short of the requirements of Lemma \ref{lem:main}, where we need a
deficit of at most \(\delta\) for any given \(\delta>0\). This is remedied by the following result, 
taking \(q \gg 1 / \delta\).

\begin{lem}\label{lem:averaging}
    Let \(t,\ell \geq 1\) be integers, let \(C > 0\), and let \(G\) be a nonempty ordered bipartite
    graph with parts \(X,W\), where \(W\) is the union of sets \(W_1 < \dots < W_q\) and \(q =
    t^{\ell}\). Suppose that the following holds: If an edge \(\xx \ww \in E(G)\) is sampled
    uniformly at random and \(\kk \in [q]\) is the unique index with \(\ww \in W_{\kk}\), then
    \(H[\kk \; | \; \xx] \geq \log q - C = \ell \log t - C\). Then there are subsets \(Y_1 <
    \dots < Y_t\) of \(W\) such that if \(\xx\yy \in E(X,Y)\) is sampled uniformly at random,
    where \(Y = Y_1 \cup \dots \cup Y_t\), and \(\jj \in [t]\) is the unique index with \(\yy
    \in Y_{\jj}\), then \(H[\jj \; | \; \xx] \geq \log t - \frac{C}{\ell}\).
\end{lem}
\begin{proof}
    We will again take nested partitions, this time partitioning \([q]\) into \(t\) parts at
    each step. More precisely, we identify \([q]\) with \([t]^\ell\) in such a way that the
    natural order on \([q]\) corresponds to the lexicographic order on \([t]^\ell\). We again
    think of \([t]^\ell\) as the set of leaves of the rooted \(t\)-ary tree of depth \(\ell\),
    and identify the nodes of this tree with the strings \(\sigma \in [t]^{\leq \ell}\). For
    each \(k \in [q]\), let \(\lambda(k) := (\lambda_1(k),\dots,\lambda_{\ell}(k)) \in
    [t]^{\ell}\) be the corresponding string.

    As in the statement of the lemma, sample an edge \(\xx\ww \in E(G)\) uniformly at random,
    and let \(\kk \in [q]\) be the unique index satisfying \(\ww \in W_{\kk}\). By the chain
    rule (Fact \ref{fact:chain rule}),
    \[
        \begin{aligned}
            \ell \log t - C
            &\leq H\big[ \kk \; \big| \; \xx \big]
            = H\big[ (\lambda_1(\kk),\dots,\lambda_\ell(\kk))
            \; \big| \; \xx \big] \\
            &= \sum_{h=1}^{\ell}
            H\big[ \lambda_h(\kk) \; \big| \;
            \lambda_1(\kk),\dots,\lambda_{h-1}(\kk),\xx \big],
        \end{aligned}
    \]
    where the inequality is by the assumptions of the lemma. By averaging, there exists \(h \in
    [\ell]\) such \nolinebreak that
    \[
        H\big[ \lambda_h(\kk) \; \big| \; \lambda_1(\kk),\dots,\lambda_{h-1}(\kk),\xx \big]
        \geq \log t - \frac{C}{\ell}.
    \]
    By the definition of conditional entropy and averaging, there exist
    \(\lambda_1,\dots,\lambda_{h-1} \in [t]\) such that
    \begin{equation}\label{eq:Y split}
        H\big[\lambda_h(\kk) \; \big| \;
        \lambda_1(\kk) = \lambda_1,\dots,\lambda_{h-1}(\kk) = \lambda_{h-1}, \; \xx \big] \geq
        \log t - \frac{C}{\ell}.
    \end{equation}

    Now, for the string \(\sigma := (\lambda_1,\dots,\lambda_{h-1})\), let
    \[
        Y := \bigcup_{k \in [q]: \; \sigma \preceq \lambda(k)}W_k,
        \qquad \qquad
        Y_i := \bigcup_{k \in [q]: \; \sigma i \preceq \lambda(k)}W_k
        \quad \text{for } i \in [t].
    \]
    So \(Y = Y_1 \cup \dots \cup Y_t\) and \(Y_1 < \dots < Y_t\), since \(W_1 < \dots < W_q\)
    and the lexicographic order on strings in \([t]^{\ell}\) corresponds to the natural order on
    \([q]\).

    Now condition on \(\ww \in Y\), and let \(\jj \in [t]\) be the unique index such that \(\ww
    \in Y_{\jj}\). Observe that \(\jj = \lambda_h(\kk)\). Conditioned on \(\ww \in Y\), the edge
    \(\xx\ww\) is distributed uniformly on \(E(X,Y)\). Moreover, by the definition of \(Y\), we
    can rewrite \eqref{eq:Y split} as
    \[
        H\big[ \jj \; \big| \;
        \ww \in Y, \; \xx \big] \geq
        \log t - \frac{C}{\ell}.
    \]
    This is exactly the conclusion of the lemma.
\end{proof}
We now combine these lemmas to prove Lemma \ref{lem:main}.
\begin{proof}[Proof of Lemma \ref{lem:main}]
    We may assume that \(0<\eps\leq1\). Let \(m_{\ref{lem:neighborhood splitting}}(s,\eps)\) be
    given by Lemma \ref{lem:neighborhood splitting}. Set
    \[
        m :=
        m_{\ref{lem:neighborhood splitting}}(s,\varepsilon/2),
        \qquad
        \ell:=\ceil*{\frac{\log m }{\delta}},
        \qquad
        q:=t^\ell,
    \]
    \[
        M := m_{\ref{lem:neighborhood splitting}}(q,\eps/2),
        \qquad
        n_0 := M^{2/\eps}.
    \]
    Since \(\ell=O_s(1/(\eps\delta))\) and \(q=t^\ell\), we have \(n_0\leq q^{O(q/\eps^2)}
    \leq2^{2^{O_{s,t}(1/(\eps\delta))}}\), as required.

    Let \(G\) be an \(n\times n\) ordered bipartite graph with parts \(U,V\), where \(n\geq
    n_0\) and \(e(G)\geq n^{1+\eps}\). Apply Lemma \ref{lem:neighborhood splitting} to \(V\)
    (i.e., with the roles of \(U,V\) switched) and with parameters \((q,\eps/2)\), obtaining
    nonempty subsets \(W_1<\dots<W_q\) of \(V\) such that, setting \(W:=W_1\cup\dots\cup W_q\),
    we have
    \[
        \sum_{u\in U}\min_{k\in[q]}d^G_{W_k}(u)
        \geq \frac{e(G)|W|}{nM}.
    \]
    Let \(G_1\subseteq G[U,W]\) be obtained by retaining, for each \(u\in U\), exactly
    \(\min_{k\in[q]}d^G_{W_k}(u)\) edges into each of the sets \(W_1,\dots,W_q\). Then
    \[
        e(G_1)
        =q\sum_{u\in U}\min_{k\in[q]}d^G_{W_k}(u)
        \geq\frac{q \cdot e(G)|W|}{nM}
        \geq n^{\eps/2}|W|,
    \]
    where the last inequality follows from \(e(G)\geq n^{1+\eps}\) and \(n^{\eps/2}\geq
    n_0^{\eps/2}=M\).

    We may therefore apply Lemma \ref{lem:neighborhood splitting} to \(U\) in \(G_1\), this time
    with parameters \((s,\eps/2)\). The lemma gives nonempty subsets \(X_1<\dots<X_s\) of \(U\)
    such that, setting \(X:=X_1\cup\dots\cup X_s\), we have
    \[
        \sum_{v\in W}\min_{i\in[s]}d^{G_1}_{X_i}(v)
        \geq \frac{e_{G_1}(X,W)}{m}.
    \]
    Once again, let \(G_2\subseteq G_1[X,W]\) be obtained by retaining, for each \(v\in W\),
    exactly \(\min_{i\in[s]}d^{G_1}_{X_i}(v)\) edges into each \(X_i\). Then
    \begin{equation}\label{eq:balanced edge retention}
        e(G_2)
        =s\sum_{v\in W}\min_{i\in[s]}d^{G_1}_{X_i}(v)
        \geq \frac{e_{G_1}(X,W)}{m}.
    \end{equation}

    Next, we show that \(G_2\) satisfies the hypothesis of Lemma \ref{lem:averaging}. Sample
    \((\xx,\ww)\in E(G_2)\) uniformly at random, and let \(\kk\in[q]\) be the index with
    \(\ww\in W_\kk\). In \(G_1\), each \(u\in X\) has exactly \(d^{G_1}(u)/q\) neighbors in each
    of the sets \(W_1,\dots,W_q\) (by the definition of \(G_1\)). Since \(G_2\subseteq
    G_1[X,W]\), for every \(u\in X\) with \(d^{G_2}(u)>0\) and every \(k\in[q]\), we have
    \[
        \bP[\kk=k\mid\xx=u]
        =\frac{d^{G_2}_{W_k}(u)}{d^{G_2}(u)}
        \leq
        \frac{d^{G_1}_{W_k}(u)}{d^{G_2}(u)}
        =
        \frac{d^{G_1}(u)}{q\,d^{G_2}(u)}.
    \]
    Thus, by Fact \ref{fact:entropy basic},
    \[
        H[\kk\mid\xx=u]
        \geq\log q-\log\frac{d^{G_1}(u)}{d^{G_2}(u)}.
    \]
    Averaging this bound over \(\xx\), using \(\bP[\xx=u]=d^{G_2}(u)/e(G_2)\) (since the edge
    \(\xx\ww\) is chosen uniformly at random), we get\footnote{Here we use the convention that
    if \(d^{G_2}(u) = 0\) then the term corresponding to \(u\) is treated as \(0\) (in this case
    the conditioning on \(\xx = u\) is undefined as \(\mathbb{P}[\xx = u] = 0\)). Furthermore,
    we treat \(0 \cdot \log (\frac{1}{0})\) as \(0\).}
    \begin{align}\label{eq:H[k | x]}
        H[\kk\mid\xx]
        &=\sum_{\substack{u\in X}}
        \frac{d^{G_2}(u)}{e(G_2)} \cdot H[\kk\mid\xx=u]
        \geq
        \log q-
        \sum_{\substack{u\in X}}
        \frac{d^{G_2}(u)}{e(G_2)}
        \cdot \log\frac{d^{G_1}(u)}{d^{G_2}(u)}.
    \end{align}
    By Jensen's inequality for the concave function \(\log\), we have
    \[
        \sum_{\substack{u\in X}}
        \frac{d^{G_2}(u)}{e(G_2)}
        \cdot \log\frac{d^{G_1}(u)}{d^{G_2}(u)}
        \leq
        \log\left(
        \sum_{\substack{u\in X}} \frac{d^{G_1}(u)}{e(G_2)} \right) =
        \log \left( \frac{e_{G_1}(X,W)}{e(G_2)} \right) \leq \log m,
    \]
    using \eqref{eq:balanced edge retention}. Combining this with \eqref{eq:H[k | x]}, we get
    \(H[\kk\mid\xx] \geq \log q - \log m\). Thus, \(G_2\) satisfies the hypothesis of Lemma
    \ref{lem:averaging} with \(C:=\log m\).

    Finally, apply Lemma \ref{lem:averaging} to \(G_2\) to obtain subsets \(Y_1<\dots<Y_t\) of
    \(W\). Set \(Y:=Y_1\cup\dots\cup Y_t\) and \(G':=G_2[X,Y]\). For a uniformly random edge
    \((\xx,\yy)\in E(G')\), let \(\ii\in[s]\) and \(\jj\in[t]\) satisfy \(\xx\in X_\ii\) and
    \(\yy\in Y_\jj\). Then
    \[
        H[\jj\mid\xx]
        \geq\log t-\frac{\log m}{\ell}
        \geq\log t-\delta
    \]
    by the choice of \(\ell\). Moreover, each \(y\in Y\) has the same number of neighbors in
    each \(X_i\) (by the definition of \(G_2\)), so \(H[\ii\mid\yy]=\log s\). Thus \(G'\subseteq
    G[X,Y]\) satisfies the required criterion, completing the proof.
\end{proof}

\paragraph{Acknowledgement.} We thank Istv\'an Tomon for reviewing the first arXiv version of
this paper and for his helpful feedback.

\paragraph{AI Disclosure.} An initial proof of Theorem \ref{thm:main} was found by ChatGPT-6
Astra through an extended conversation with the authors. The authors rewrote this proof for the
first arXiv version of this paper. Further conversations led to the reformulation presented in
Section \ref{sec:main-lemmas} and an alternative proof of Lemma \ref{lem:main}. Based on that
proof, the authors developed a substantially simplified proof in Section
\ref{sec:proof-main-lemma} without AI assistance. The authors take full responsibility for the
correctness of the paper.

\fontsize{11pt}{12pt}
\selectfont
	
\hypersetup{linkcolor={red!70!black}}
\setlength{\parskip}{2pt plus 0.3ex minus 0.3ex}

\bibliographystyle{auxfile.bst}
\bibliography{bib.bib}

@article{ConlonKimLeeLee2018,
  author  = {Conlon, David and Kim, Jeong Han and Lee, Choongbum
             and Lee, Joonkyung},
  title   = {Some advances on {Sidorenko}'s conjecture},
  journal = {J. Lond. Math. Soc. (2)},
  volume  = {98},
  number  = {3},
  pages   = {593--608},
  year    = {2018},
  doi     = {10.1112/jlms.12142}
}

@article{Furedi1990,
  author  = {F{\"u}redi, Zolt{\'a}n},
  title   = {The maximum number of unit distances in a convex $n$-gon},
  journal = {J. Combin. Theory Ser. A},
  volume  = {55},
  number  = {2},
  pages   = {316--320},
  year    = {1990},
  doi     = {10.1016/0097-3165(90)90074-7}
}

@article{FurediHajnal1992,
  author  = {F{\"u}redi, Zolt{\'a}n and Hajnal, P{\'e}ter},
  title   = {{Davenport--Schinzel} theory of matrices},
  journal = {Discrete Math.},
  volume  = {103},
  number  = {3},
  pages   = {233--251},
  year    = {1992},
  doi     = {10.1016/0012-365X(92)90316-8}
}

@article{JanzerJanzerMagnanMethuku2024,
  author  = {Janzer, Barnab{\'a}s and Janzer, Oliver and Magnan, Van
             and Methuku, Abhishek},
  title   = {Tight general bounds for the extremal numbers of
             0--1 matrices},
  journal = {Int. Math. Res. Not. IMRN},
  volume  = {2024},
  number  = {15},
  pages   = {11455--11463},
  year    = {2024},
  doi     = {10.1093/imrn/rnae129}
}

@article{JJMT,
  author  = {Janzer, Barnab{\'a}s and Janzer, Oliver and Methuku, Abhishek
             and Tardos, G{\'a}bor},
  title   = {Tight bounds for intersection-reverse sequences,
             edge-ordered graphs, and applications},
  journal = {J. Lond. Math. Soc. (2)},
  volume  = {112},
  number  = {4},
  pages   = {e70324},
  year    = {2025},
  doi     = {10.1112/jlms.70324}
}

@article{Keszegh2009,
  author  = {Keszegh, Bal{\'a}zs},
  title   = {On linear forbidden submatrices},
  journal = {J. Combin. Theory Ser. A},
  volume  = {116},
  number  = {1},
  pages   = {232--241},
  year    = {2009},
  doi     = {10.1016/j.jcta.2008.05.006}
}

@inproceedings{Klazar2000,
  author    = {Klazar, Martin},
  title     = {The {F{\"u}redi--Hajnal} conjecture implies the
               {Stanley--Wilf} conjecture},
  booktitle = {Formal Power Series and Algebraic Combinatorics},
  editor    = {Krob, Daniel and Mikhalev, Alexander A.
               and Mikhalev, Alexander V.},
  publisher = {Springer},
  address   = {Berlin},
  pages     = {250--255},
  year      = {2000},
  doi       = {10.1007/978-3-662-04166-6_22}
}

@article{KorandiTardosTomonWeidert2019,
  author  = {Kor{\'a}ndi, D{\'a}niel and Tardos, G{\'a}bor
             and Tomon, Istv{\'a}n and Weidert, Craig},
  title   = {On the {Tur{\'a}n} number of ordered forests},
  journal = {J. Combin. Theory Ser. A},
  volume  = {165},
  pages   = {32--43},
  year    = {2019},
  doi     = {10.1016/j.jcta.2019.01.006}
}

@article{KucheriyaTardos2023,
  author  = {Kucheriya, Gaurav and Tardos, G{\'a}bor},
  title   = {A characterization of edge-ordered graphs with
             almost linear extremal functions},
  journal = {Combinatorica},
  volume  = {43},
  number  = {6},
  pages   = {1111--1123},
  year    = {2023},
  doi     = {10.1007/s00493-023-00052-5}
}

@article{MarcusTardos2004,
  author  = {Marcus, Adam and Tardos, G{\'a}bor},
  title   = {Excluded permutation matrices and the
             {Stanley--Wilf} conjecture},
  journal = {J. Combin. Theory Ser. A},
  volume  = {107},
  number  = {1},
  pages   = {153--160},
  year    = {2004},
  doi     = {10.1016/j.jcta.2004.04.002}
}

@article{PachTardos2006,
  author  = {Pach, J{\'a}nos and Tardos, G{\'a}bor},
  title   = {Forbidden paths and cycles in ordered graphs and matrices},
  journal = {Israel J. Math.},
  volume  = {155},
  number  = {1},
  pages   = {359--380},
  year    = {2006},
  doi     = {10.1007/BF02773960}
}

@article{Pettie2011a,
  author  = {Pettie, Seth},
  title   = {Degrees of nonlinearity in forbidden 0--1 matrix problems},
  journal = {Discrete Math.},
  volume  = {311},
  number  = {21},
  pages   = {2396--2410},
  year    = {2011},
  doi     = {10.1016/j.disc.2011.06.020}
}

@article{Pettie2011b,
  author  = {Pettie, Seth},
  title   = {Generalized {Davenport--Schinzel} sequences and
             their 0--1 matrix counterparts},
  journal = {J. Combin. Theory Ser. A},
  volume  = {118},
  number  = {6},
  pages   = {1863--1895},
  year    = {2011},
  doi     = {10.1016/j.jcta.2011.02.011}
}

@inproceedings{PettieTardos2024,
  author    = {Pettie, Seth and Tardos, G{\'a}bor},
  title     = {On the extremal functions of acyclic forbidden
               0--1 matrices},
  booktitle = {Proceedings of the 2024 Annual {ACM-SIAM} Symposium
               on Discrete Algorithms ({SODA})},
  editor    = {Woodruff, David P.},
  publisher = {SIAM},
  address   = {Philadelphia, PA},
  pages     = {1166--1176},
  year      = {2024},
  doi       = {10.1137/1.9781611977912.45}
}

@article{PettieTardos2025,
  author  = {Pettie, Seth and Tardos, G{\'a}bor},
  title   = {A refutation of the {Pach--Tardos} conjecture
             for 0--1 matrices},
  journal = {Combinatorica},
  volume  = {45},
  number  = {6},
  year    = {2025},
  note    = {Paper No.~59, 30 pp.},
  doi     = {10.1007/s00493-025-00187-7}
}

@inproceedings{Tardos2019a,
  author    = {Tardos, G{\'a}bor},
  title     = {Extremal theory of ordered graphs},
  booktitle = {Proceedings of the International Congress of
               Mathematicians ({ICM} 2018)},
  editor    = {Sirakov, Boyan and de Souza, Paulo Ney
               and Viana, Marcelo},
  volume    = {4},
  pages     = {3235--3243},
  publisher = {World Scientific},
  address   = {Singapore},
  year      = {2019},
  doi       = {10.1142/9789813272880_0179}
}

@incollection{Tardos2019b,
  author    = {Tardos, G{\'a}bor},
  title     = {Extremal theory of vertex or edge ordered graphs},
  booktitle = {Surveys in Combinatorics 2019},
  editor    = {Lo, Allan and Mycroft, Richard and Perarnau, Guillem
               and Treglown, Andrew},
  series    = {London Math. Soc. Lecture Note Ser.},
  volume    = {456},
  pages     = {221--236},
  publisher = {Cambridge University Press},
  address   = {Cambridge},
  year      = {2019},
  doi       = {10.1017/9781108649094.008}
}

@book{Zhao2023,
  author    = {Zhao, Yufei},
  title     = {Graph Theory and Additive Combinatorics:
               Exploring Structure and Randomness},
  publisher = {Cambridge University Press},
  address   = {Cambridge},
  year      = {2023},
  doi       = {10.1017/9781009310956}
}

\end{document}